\documentclass[12pt]{amsart}
\usepackage{amsmath,amssymb}
\usepackage{calrsfs}
\newcommand{\Z}{{\mathbb{Z}}}
\newcommand{\Q}{{\mathbb{Q}}}
\newcommand{\N}{{\mathbb{N}}}

\newcommand{\C}{{\mathbb{C}}}
\newcommand{\gl}{{\mathfrak{gl}}}
\newcommand{\slm}{{\mathfrak{sl}}}
\newcommand{\fg}{{\mathfrak{g}}}
\newcommand{\fh}{{\mathfrak{h}}}
\newcommand{\fn}{{\mathfrak{n}}}
\newcommand{\bi}{{\mathbf{i}}}
\newcommand{\be}{{\mathbf{e}}}

\newcommand{\cO}{{\mathcal{O}}}

\newcommand{\dalp}{{\dot{\alpha}}}
\newcommand{\dbet}{{\dot{\beta}}}
\newcommand{\dPhi}{{\dot{\Phi}}}
\newcommand{\dPi}{{\dot{\Pi}}}

\newcommand{\GL}{{\operatorname{GL}}}
\newcommand{\hgt}{{\operatorname{ht}}}
\renewcommand{\leq}{\leqslant}
\renewcommand{\geq}{\geqslant}

\newtheorem{thm}{Theorem}[section]
\newtheorem{cor}[thm]{Corollary}
\newtheorem{lem}[thm]{Lemma}
\newtheorem{prop}[thm]{Proposition}
\theoremstyle{definition}
\newtheorem{defn}[thm]{Definition}
\newtheorem{exmp}[thm]{Example}
\newtheorem{rema}[thm]{}
\theoremstyle{remark}
\newtheorem{rem}[thm]{Remark}
\numberwithin{equation}{section}

\begin{document}

\title{On the construction of semisimple Lie algebras and Chevalley groups}
\author{Meinolf Geck}
\address{IAZ - Lehrstuhl f\"ur Algebra\\Universit\"at Stuttgart\\ 
Pfaffenwaldring 57\\D--70569 Stuttgart\\ Germany}
\email{meinolf.geck@mathematik.uni-stuttgart.de}

\date{}

\begin{abstract} Let $\fg$ be a semisimple complex Lie algebra. Recently,
Lusztig simplified the traditional construction of the corresponding 
Chevalley groups (of adjoint type) using the ''canonical basis'' of the 
adjoint representation of~$\fg$. Here, we present a variation of this 
idea which leads to a new, and quite elementary construction of~$\fg$ 
itself from its root system. An additional feature of this set-up 
is that it also gives rise to explicit Chevalley bases of $\fg$.
\end{abstract}

\maketitle

\section{Introduction} \label{sec0}

Let $\fg$ be a finite-dimensional semisimple Lie algebra over $\C$. In
a famous paper \cite{Ch}, Chevalley found an integral basis of $\fg$ and 
used this to construct new families of simple groups, now known as
{\em Chevalley groups}. In \cite{L4}, Lusztig described a simplified 
construction of these groups, by using a remarkable basis of the adjoint 
representation of $\fg$ on which the Chevalley generators $e_i, f_i \in 
\fg$ act via matrices with entries in $\N_0$. That basis originally
appeared in \cite{L1}, \cite{L2}, \cite{L2a}, even at the quantum group 
level;  subsequently, it could be interpreted as the {\em canonical basis} 
of the adjoint representation (see \cite{L6}, \cite{L5}). 

In this note, we turn Lusztig's argument around and show that this leads to 
a new way of actually constructing~$\fg$ from its root system. 
Usually, this is achieved by a subtle choice of signs in 
Chevalley's integral basis (Tits \cite{Ti}), or by taking a suitable 
quotient of a free Lie algebra (Serre \cite{S}). In our approach, we do not 
need to choose any signs, and we do not have to deal with free Lie 
algebras at all. 

Section~\ref{sec1} contains some preliminary results about root strings in 
a root system~$\Phi$. In Section~\ref{sec2}, we use the explicit formulae 
in \cite{L1}, \cite{L2}, \cite{L2a} to define certain linear maps $e_i,f_i$
on a finite-dimensional vector space. These maps satisfy relations 
which are to be expected of the Chevalley generators of a semisimple Lie 
algebra. In Section~\ref{sec3}, it is shown that $e_i,f_i$ indeed generate 
a semisimple Lie algebra $\fg$ with root system isomorphic to~$\Phi$. This 
set-up also gives rise to two Chevalley bases of $\fg$ which are explicitly 
determined by the two ``canonical'' orientations of the Dynkin diagram 
of~$\Phi$ in which every vertex is either a sink or a source; see 
Section~\ref{sec4}. (This seems conceptually simpler than the approach
via Hall algebras and the representation theory of quivers, see 
Ringel \cite{Ri} and Peng--Xiao \cite{PX}.)

We have made a certain attempt to keep the whole discussion as elementary
and self-contained as possible; it should be accessible at the level of 
introductory textbooks on Lie algebras, e.g., Erdmann--Wildon \cite{EW}.

Finally, we remark that Lusztig's simplified construction of adjoint 
Chevalley groups can be extended to groups which are not necessarily of 
adjoint type, avoiding the additional machinery required in \cite{Ch2}, 
\cite{St}; for further details see \cite{G1}. (See Lusztig \cite{L3} for a 
more sophisticated setting, which yields additional results and also 
produces arbitrary reductive algebraic groups.)

Our results are merely variations of ideas in \cite{L1}, \cite{L2},
\cite{L2a}, \cite{L4}. I~wish to thank George Lusztig for pointing out 
to me the remarks in~\cite{L4}, and for helpful comments about his earlier 
work on quantum groups and the interpretation of $e_i,f_i$ in terms of 
canonical bases. I also thank Lacri Iancu for a careful reading of the 
mansucript, Markus Reineke for pointing out Ringel's paper \cite{Ri},
and an unknown referee for a number of useful comments.

\section{Root systems and root strings} \label{sec1}

We begin by recalling basic facts about root systems; see \cite{B}, 
\cite{EW}, \cite{H}, \cite{S}. Let $E$ be a finite-dimensional vector space 
over $\Q$ and $(\;,\;) \colon E\times E\rightarrow \Q$ be a symmetric 
bilinear form such that $(e,e)>0$ for all $0\neq e\in E$. For each $0\neq 
e\in E$, we denote $e^\vee:=\frac{2}{(e,e)}e\in E$. Let $\Phi\subseteq E$ be 
a reduced crystallographic root system. Thus, $\Phi$ is a finite subset of 
$E\setminus \{0\}$ such that $E=\langle \Phi \rangle_\Q$; furthermore, 
we have for $\alpha,\beta\in \Phi$:
\begin{itemize}
\item if $\beta\neq \pm \alpha$, then $\alpha,\beta$ are linearly
independent in $E$;
\item we have $(\beta,\alpha^\vee)\in\Z$ and $\beta-(\beta,\alpha^\vee)
\alpha \in \Phi$.
\end{itemize}
We assume throughout that $\Phi$ is irreducible. Let $\Pi=\{\alpha_i\mid 
i\in I\}$ be a set of simple roots in $\Phi$, where $I$ is a finite index 
set. (One may take for $I$ the ``canonical index set'' in \cite[VI, no.~1.5,
Remarque~7]{B}.) Then $\Pi$ is a basis of $E$ and every $\alpha\in \Phi$ is 
a linear combination of $\Pi$ where either all coefficients are in 
$\Z_{\geq 0}$ or all coefficients are in $\Z_{\leq 0}$; correspondingly, we 
have a partition $\Phi=\Phi^+\cup \Phi^-$ where $\Phi^+$ are the positive 
roots and $\Phi^-=-\Phi^+$ are the negative roots. The matrix 
\[A=(a_{ij})_{i,j\in I}\qquad \mbox{where} \qquad a_{ij}:=(\alpha_j,
\alpha_i^\vee),\]
is called the Cartan matrix of $\Phi$ with respect to $\Pi$. We have
$a_{ii}=2$ and $a_{ij}\leq 0$ for $i\neq j$. Furthermore, it is known
that $A$ is independent of the choice of $\Pi$, up to simultaneous 
permutation of the rows and columns.  

\begin{rem} \label{rem1} Using the Cauchy--Schwartz inequality, the
above conditions immediately imply the following ''finiteness property'':
\begin{equation*}
\beta\neq\pm \alpha\qquad\Rightarrow\qquad (\beta,\alpha^\vee)(\alpha,
\beta^\vee) \in\{0,1,2,3\}.\tag{a}
\end{equation*}
Furthermore, if the value $3$ is attained for some $\alpha,\beta\in\Phi$,
then $\dim E=2$. (See \cite[VI, \S 4, no.~4.1]{B}.) There are a number of
constraints on the relative lengths of roots. For example, if $\beta 
\neq \pm\alpha$ and $(\alpha,\beta^\vee)=\pm 1$, then we must have 
$(\alpha,\alpha)\leq (\beta,\beta)$ (since $(\beta,\alpha^\vee)\in\Z$). 
Furthermore, we have (see \cite[VI, \S 1, no.~1.4]{B}): 
\begin{equation*}
|\{(\alpha,\alpha)\mid \alpha\in\Phi\}|\leq 2. \tag{b}
\end{equation*}
\end{rem} 

\begin{rema} {\bf Root strings.} \label{strgs} Given $\alpha, \beta\in \Phi$ 
such that $\beta \neq \pm \alpha$, we can uniquely define two integers 
$p,q\geq 0$ by the conditions that 
\[ \beta-q\alpha,\ldots,\beta-\alpha,\beta,\beta+\alpha,\ldots,
\beta+p\alpha\in \Phi\]
and $\beta-(q+1)\alpha\not\in\Phi$, $\beta+(p+1)\alpha\not\in\Phi$. The 
above sequence of roots is called the {\em $\alpha$-string through~$\beta$}. 
We then have (see \cite[VI, \S 1, no.~1.3]{B}):
\[ (\beta,\alpha^\vee)=q-p\qquad \mbox{and}\qquad 0\leq p+q\leq 3.\]
Furthermore, if $p+q=3$, then $(\beta-q\alpha,\alpha^\vee)=-p-q=-3$ and 
so we must have $\dim E=2$ (see Remark~\ref{rem1}).
\end{rema}

\begin{defn} \label{def1} In the above setting,
we let $m_\alpha^+(\beta):=p+1$ and $m_\alpha^-(\beta):=q+1$. 
Then $(\beta,\alpha^\vee)=m_\alpha^-(\beta)-m_\alpha^+(\beta)$ and 
$2\leq m_\alpha^-(\beta)+m_\alpha^+(\beta)\leq 5$. We also write
$m_i^\pm(\beta):=m^{\pm}_{\alpha_i}(\beta)$ if $\alpha=\alpha_i$ with 
$i \in I$.
\end{defn}

We now collect some results which will be useful in the sequel. 

\begin{lem} \label{lem1} Let $\alpha,\beta\in\Phi$. If $(\alpha,
\beta)>0$, then $\beta-\alpha\in\Phi$ or $\beta=\alpha$.
\end{lem}

\begin{proof} See \cite[VI, \S 1, no.~1.3]{B}, 
\cite[\S 11.2]{EW} or \cite[\S 9.4]{H}. 
\end{proof}

\begin{lem} \label{lem2} Let $\alpha\in\Phi$ and $i,j\in I$, $i\neq j$.
Assume that $\alpha+\alpha_i-\alpha_j\in\Phi$.

{\rm (a)} If $\alpha+\alpha_i\in\Phi$ and $\alpha\neq\alpha_j$, then 
$\alpha-\alpha_j\in\Phi$.

{\rm (b)} If $\alpha-\alpha_j\in\Phi$ and $\alpha\neq -\alpha_i$, 
then $\alpha+\alpha_i\in\Phi$.
\end{lem}

\begin{proof} (a) If $m:=(\alpha_j, \alpha^\vee)>0$, then $\alpha-
\alpha_j \in\Phi$ by Lemma~\ref{lem1}. Now assume that $m\leq 0$ 
and set $\beta:=\alpha+\alpha_i-\alpha_j$. Then $\beta-\alpha\not\in\Phi$ 
and $\beta\neq \alpha$. Hence, by Lemma~\ref{lem1}, we have $0\geq (\beta,
\alpha^\vee)=2+(\alpha_i, \alpha^\vee)-m$ and so $(\alpha_i,\alpha^\vee) 
\leq -2$. Since $\alpha\neq \pm\alpha_i$, the condition in 
Remark~\ref{rem1}(a) implies that $(\alpha, \alpha_i^\vee)=-1$ and so 
$(\beta,\alpha_i^\vee)=1-(\alpha_j, \alpha_i^\vee)\geq 1$. Hence, 
$\beta-\alpha_i=\alpha-\alpha_j \in\Phi$, again by Lemma~\ref{lem1}.

(b) Apply (a) to $-\alpha+\alpha_j-\alpha_i\in\Phi$, exchanging the roles
of $\alpha_i$, $\alpha_j$.
\end{proof}

\begin{lem} \label{lem3} Let $\alpha\in\Phi$ and $i,j\in I$, $i\neq j$.
Assume that $\alpha+\alpha_i\in\Phi$, $\alpha-\alpha_j \in\Phi$ and 
$\alpha+\alpha_i-\alpha_j\in\Phi$. Then $m_i^-(\alpha-\alpha_j) 
m_j^+(\alpha)=m_j^+(\alpha+\alpha_i)m_i^-(\alpha)$. (Note that the
assumptions imply that $\alpha\neq \pm\alpha_i$, $\alpha\neq \pm\alpha_j$, 
$\alpha+\alpha_i\neq \pm\alpha_j$, $\alpha-\alpha_j\neq \pm\alpha_i$.) 
\end{lem}

\begin{proof} It is easy to check (by an explicit verification for types
$A_2$, $B_2$, $G_2$) that the above assumptions can not be satisfied if
$\dim E=2$. So assume now that $\dim E\geq 3$. Then, by Remark~\ref{rem1}(a), 
we have $(\beta,\gamma^\vee)\in\{0,\pm 1, \pm 2\}$ for all $\beta,
\gamma \in\Phi$, and so every root string in $\Phi$ can have at most 
$3$ terms.  

Now, we have $\alpha,\alpha+\alpha_i\in \Phi$ and so $\alpha\neq\pm\alpha_i$. 
Hence, the $\alpha_i$-string through $\alpha$ must be $(\alpha,
\alpha+\alpha_i)$ or $(\alpha-\alpha_i, \alpha,\alpha+\alpha_i)$ or 
$(\alpha,\alpha+\alpha_i,\alpha+2\alpha_i)$. This yields:
\[(\alpha,\alpha_i^\vee)\leq 0\qquad \mbox{and}\qquad m_i^-(\alpha)=
\left\{\begin{array}{cl} 2 &\mbox{ if $(\alpha,
\alpha_i^\vee)=0$},\\ 1 &\mbox{ otherwise}.\end{array}\right.\]
Similarly, we have $\alpha-\alpha_j,\alpha\in\Phi$ and $\alpha\neq \pm
\alpha_j$; this yields:
\[(\alpha,\alpha_j^\vee)\geq 0\qquad \mbox{and}\qquad 
m_j^+(\alpha)=\left\{\begin{array}{cl} 2 & \mbox{ if $(\alpha,
\alpha_j^\vee)=0$},\\ 1 & \mbox{ otherwise}.\end{array}\right.\]
Furthermore, we have $\alpha-\alpha_j,\alpha-\alpha_j+\alpha_i\in \Phi$ and
$\alpha-\alpha_j\neq \pm\alpha_i$; hence, 
\[(\alpha-\alpha_j,\alpha_i^\vee)\leq 0\quad \mbox{and}\quad 
m_i^-(\alpha-\alpha_j)=\left\{\begin{array}{cl} 2 & \mbox{ if
$(\alpha-\alpha_j,\alpha_i^\vee)=0$},\\ 1 & \mbox{ otherwise},
\end{array} \right.\]
Finally, we have $\alpha+\alpha_i-\alpha_j,\alpha+\alpha_i\in\Phi$ and
$\alpha+\alpha_i\neq \pm\alpha_j$; hence, 
\[(\alpha+\alpha_i,\alpha_j^\vee)\geq 0\quad \mbox{and}\quad 
m_j^+(\alpha+\alpha_i)=\left\{\begin{array}{cl} 2 & \mbox{ if
$(\alpha+\alpha_i,\alpha_j^\vee)=0$},\\ 1 & \mbox{ otherwise}.
\end{array} \right.\]
We now distinguish cases according to when one of the above terms equals~$2$. 

\smallskip
\noindent {\em Case 1}. Assume that $m_i^-(\alpha)=2$ and so $(\alpha,
\alpha_i^\vee)=0$. Since $\alpha_i- \alpha_j\not\in\Phi$, we have 
$(\alpha+\alpha_i-\alpha_j, \alpha^\vee)\leq 0$ and so $(\alpha_j,
\alpha^\vee)\geq 2$, hence $m_j^+(\alpha)=1$. Using Remark~\ref{rem1}(a),
we conclude that $(\alpha_j,\alpha^\vee)=2$ and $(\alpha,\alpha_j^\vee)=1$.
Furthermore, $0\geq (\alpha-\alpha_j,\alpha_i^\vee)=-(\alpha_j,\alpha_i^\vee)
\geq 0$ and so $(\alpha_j,\alpha_i^\vee)=0$; hence, $(\alpha-
\alpha_j, \alpha_i^\vee)=0$ and $(\alpha+\alpha_i,\alpha_j^\vee)=1$. 
Then $m_i^-(\alpha-\alpha_j)=2$ and $m_j^+(\alpha+\alpha_i)=1$, so the 
desired identity holds.

\smallskip
\noindent {\em Case 2}. Assume that $m_j^+(\alpha)=2$ and so $(\alpha,
\alpha_j^\vee)=0$. Then we can apply Case~1 to $-\alpha+\alpha_j-\alpha_i$
(with the roles of $\alpha_i$, $\alpha_j$ exchanged). Consequently, we 
obtain that $m_j^-(-\alpha-\alpha_i)m_i^+(-\alpha)=m_i^+(-\alpha+\alpha_j)
m_j^-(-\alpha)$. It remains to use the fact that $m^-_\beta(\gamma)=
m_{\beta}^+(-\gamma)$ for all $\beta,\gamma\in \Phi$ such that 
$\gamma\neq \pm \beta$. 

\smallskip
\noindent {\em Case 3}. We can now assume that $m_i^-(\alpha)=
m_j^+(\alpha)=1$; then $(\alpha,\alpha_i^\vee)<0$ and $(\alpha,
\alpha_j^\vee)>0$. We must show that $m_j^+(\alpha+\alpha_i)=
m_i^-(\alpha-\alpha_j)$. 

Assume, if possible, that $m_i^-(\alpha-\alpha_j)=2$ and $m_j^+(\alpha+
\alpha_i)=1$. This means that $(\alpha-\alpha_j,\alpha_i^\vee)=0$ and so
$(\alpha_j,\alpha_i^\vee)=(\alpha,\alpha_i^\vee)<0$; furthermore, $(\alpha+
\alpha_i,\alpha_j^\vee)>0$. Hence, $0<-(\alpha_i,\alpha_j^\vee)<
(\alpha, \alpha_j^\vee)\leq 2$ and so $(\alpha_i,\alpha_j^\vee)=-1$,
$(\alpha,\alpha_j^\vee)=2$. Using Remark~\ref{rem1}(a), we conclude that
$(\alpha,\alpha)>(\alpha_j,\alpha_j)\geq (\alpha_i,\alpha_i)$; but then
Remark~\ref{rem1}(b) implies that $(\alpha_i,\alpha_i)=(\alpha_j,\alpha_j)$.
This yields $(\alpha,\alpha_i^\vee)=(\alpha_j, \alpha_i^\vee)=(\alpha_i, 
\alpha_j^\vee)=-1$, which in turn implies that $(\alpha,\alpha)\leq 
(\alpha_i,\alpha_i)$, a contradiction.

On the other hand, if we assume that $m_i^-(\alpha-\alpha_j)=1$ and 
$m_j^+(\alpha+ \alpha_i)=2$, then we can apply the previous argument to 
$-\alpha+\alpha_j-\alpha_i$ (cf.\ Case~2) and, again, obtain a contradiction.
\end{proof}

\begin{rem} The above proof simplifies drastically in the simply-laced
case. (The ``ADE'' case.) In this case, every root string has at most 
$2$ terms, and so we have $m_i^-(\alpha-\alpha_j)=m_j^+(\alpha)=
m_j^+(\alpha+ \alpha_i)=m_i^-(\alpha)=1$.
\end{rem}

\section{A canonical model for the adjoint representation} \label{sec2}

We keep the notation of the previous section, where $\Phi$ is a root system 
in $E$ and $\Pi=\{\alpha_i\mid i\in I\}$ is a fixed set of simple roots;
recall that $\alpha^\vee=\frac{2}{(\alpha,\alpha)}\alpha\in E$ for $\alpha
\in\Phi$. Following Lusztig \cite[1.4]{L1}, \cite[2.1]{L2}, \cite[0.5]{L2a}, 
we now consider a $\C$-vector space $M$ with a basis $\{u_i\mid i\in I\}
\cup \{v_\alpha\mid \alpha \in\Phi\}$ and define linear maps 
\[e_i\colon M\rightarrow M,\qquad f_i\colon M\rightarrow M,\qquad
h_i\colon M\rightarrow M\qquad (i\in I)\]
by the following formulae, where $j\in I$ and $\alpha\in\Phi$.
\begin{alignat*}{2}
e_i(u_j) &:= |(\alpha_i,\alpha_j^\vee)|v_{\alpha_i},& \qquad e_i(v_\alpha) 
&:= \left\{\begin{array}{cl} m_i^-(\alpha)
v_{\alpha+\alpha_i}  & \mbox{ if $\alpha+\alpha_i \in\Phi$},\\
u_i & \mbox{ if $\alpha=-\alpha_i$},\\ 0  & \mbox{ otherwise},
\end{array}\right.\\ f_i(u_j) &:= |(\alpha_i,\alpha_j^\vee)|
v_{-\alpha_i},& \qquad f_i(v_\alpha)& :=\left\{\begin{array}{cl} 
m_i^+(\alpha) v_{\alpha-\alpha_i}  & \mbox{ if $\alpha-\alpha_i 
\in\Phi$},\\ u_i & \mbox{ if $\alpha=\alpha_i$},\\ 0 & \mbox{ otherwise}, 
\end{array}\right.\\ h_i(u_j) &:=0, &\qquad h_i(v_\alpha)&:=(\alpha,
\alpha_i^\vee) v_\alpha.
\end{alignat*}
Recall from Definition~\ref{def1} that $m_i^\pm(\alpha)=
m^{\pm}_{\alpha_i}(\alpha)\geq 1$. Note that all entries of the matrices of 
$e_i$, $f_i$ with respect to the basis of $M$ are non-negative integers. 

In \cite{L1}, \cite{L2} (see also \cite[1.15]{L5}) it is shown that the 
adjoint representation of a semisimple Lie algebra $\fg$ admits a basis on 
which the Chevalley generators of $\fg$ act via the above maps; hence,
$e_i,f_i, h_i$ must satisfy certain relations. We now verify directly, i.e.,
without reference to $\fg$, that these relations hold. 

A similar verification can also be found in \cite[5A.5]{J}; however, full 
details for the most difficult case (which is Case~2 in Lemma~\ref{lem24}) 
are not given there. In any case, since \cite[5A.5]{J} also deals with the 
quantum group case, the argument here is technically much simpler, so we give 
the details below.

\begin{rem} \label{rem20} It is obvious that the maps $e_i,f_i,h_i$ are 
all non-zero. With respect to the given basis of $M$, each $h_i$ is 
represented by a diagonal matrix, so it is clear that $h_i\circ h_j=
h_j\circ h_i$ for all $i,j$. Furthermore, the maps $h_i$ ($i\in I$) are 
actually linearly independent. Indeed, assume we have a relation 
$\sum_{j\in I} x_jh_j=0$ where $x_j\in \C$. Evaluating this relation at
$v_{\alpha_i}$ yields $(\sum_{j\in I} x_j(\alpha_i,\alpha_j^\vee)) 
v_{\alpha_i}=0$ for all $i\in I$. Since the Cartan matrix $A$ is 
invertible, we must have $x_j=0$ for all $j\in I$.
\end{rem}

\begin{rem} \label{lem20} We define a further linear map $\omega\colon 
M\rightarrow M$ by $\omega(u_j)=u_j$ for $j\in I$ and $\omega(v_\alpha)=
v_{-\alpha}$ for $\alpha\in \Phi$; note that $\omega^2=\mbox{id}_M$. 
Then one easily checks that
\[ \omega \circ e_i=f_i\circ \omega\quad\mbox{and}\quad 
\omega\circ h_i=-h_i \circ \omega \qquad \mbox{for all $i\in I$}.\]
(Just note that $m_i^-(\alpha)=m_i^+(-\alpha)$ for all 
$\alpha\in\Phi$ and $i\in I$ such that $\alpha\neq\pm\alpha_i$.) 
\end{rem}

\begin{lem} \label{lem22} {\rm (a)} We have  $h_j\circ e_i-e_i\circ h_j=
(\alpha_i,\alpha_j^\vee)e_i$ for all $i,j$.

{\rm (b)} We have $h_j\circ f_i-f_i\circ h_j= -(\alpha_i,\alpha_j^\vee)f_i$ 
for all $i,j$.
\end{lem}

\begin{proof} First we prove (a). For any $k\in I$, we have 
\[  (h_j\circ e_i-e_i\circ h_j)(u_k)=|(\alpha_i,\alpha_k^\vee)|h_j
(v_{\alpha_i})=|(\alpha_i,\alpha_k^\vee)|(\alpha_i,\alpha_j^\vee) 
v_{\alpha_i},\]
which is the same as $(\alpha_i,\alpha_j^\vee)e_i(u_k)$.
For any $\alpha\in \Phi$, we obtain
\[(h_j\circ e_i-e_i\circ h_j)(v_\alpha)=h_j(e_i(v_\alpha))-(\alpha,
\alpha_j^\vee)e_i(v_\alpha).\]
Now, if $\alpha+\alpha_i\not\in\Phi$, then the result is $0$, and this 
is also the result of $(\alpha, \alpha_j^\vee)e_i(v_\alpha)$. If 
$\alpha=-\alpha_i$, then the result is $h_j(u_i)-(\alpha,\alpha_j^\vee)u_i=
(\alpha_i,\alpha_j^\vee)u_i$, which is also the result of $(\alpha_i,
\alpha_j^\vee)e_i(v_\alpha)$. Finally, if $\alpha+\alpha_i\in\Phi$, then 
\begin{align*}
(h_j\circ e_i-e_i\circ h_j)(v_\alpha)&=m_i^-(\alpha)(\alpha+\alpha_i,
\alpha_j^\vee)v_{\alpha+\alpha_i} -(\alpha,\alpha_j^\vee)m_i^-(\alpha)
v_{\alpha+\alpha_i}\\
&=m_i^-(\alpha)(\alpha_i,\alpha_j^\vee)v_{\alpha+\alpha_i} 
=(\alpha_i,\alpha_j^\vee)e_i(v_\alpha),
\end{align*}
as required. Then (b) follows using the map $\omega\colon M\rightarrow
M$ in Remark~\ref{lem20}.
\end{proof}

\begin{lem} \label{lem23} We have $e_i\circ f_i-f_i\circ e_i=h_i$ for all~$i$.
\end{lem}

\begin{proof} For $k\in I$, we have $h_i(u_k)=0$ and $e_i(f_i(u_k))=
|(\alpha_i,\alpha_k^\vee)|u_i=f_i(e_i(u_k))$, as required. Now let 
$\alpha\in\Phi$. Then we must show that 
\begin{equation*}
e_i(f_i(v_\alpha))-f_i (e_i(v_\alpha))=(\alpha,\alpha_i^\vee) v_\alpha.
\tag{$*$}
\end{equation*}
If $\alpha=\alpha_i$, then both sides of ($*$) are equal to 
$2v_{\alpha_i}$. Similary, if $\alpha=-\alpha_i$, then both sides 
are equal to $-2v_{\alpha_i}$. Now assume that $\alpha\neq \pm \alpha_i$.

If $\alpha+\alpha_i\not\in\Phi$ and $\alpha-\alpha_i\not\in\Phi$, then 
$e_i(v_\alpha)=f_i(v_\alpha)=0$ and so the left hand side of ($*$) is~$0$.
On the other hand, the right hand side also equals~$0$, since 
$m_i^\pm(\alpha)=1$ and so $(\alpha,\alpha_i^\vee)=m_i^-(\alpha)-
m_i^+(\alpha)=0$.

If $\alpha+\alpha_i\in\Phi$ and $\alpha-\alpha_i\not \in\Phi$, then 
$m_i^-(\alpha)=1$, $f_i(v_\alpha)=0$, $e_i(v_\alpha)=v_{\alpha+\alpha_i}$ 
and $f_i(v_{\alpha+\alpha_i})=m_i^+(\alpha+\alpha_i)v_\alpha$. Furthermore,
$m_i^+(\alpha+\alpha_i)=m_i^+(\alpha)-1$ and so the left hand side of ($*$)
equals $(-m_i^+(\alpha)+1)v_\alpha$. Since $(\alpha,\alpha_i^\vee)=
m_i^-(\alpha)-m_i^+(\alpha)$, this also equals the right hand side of~($*$).

If $\alpha+\alpha_i\not\in\Phi$ and $\alpha-\alpha_i\in\Phi$, then the
argument is completely analogous to the previous case. 
Finally, if $\alpha\pm \alpha_i\in \Phi$, then the left hand side of 
($*$) equals 
\[m_i^+(\alpha)m_i^-(\alpha-\alpha_i)v_\alpha-m_i^-(\alpha)m_i^+
(\alpha+ \alpha_i)v_\alpha.\]
Now $m_i^-(\alpha-\alpha_i)=m_i^-(\alpha)-1$ and $m_i^+(\alpha+\alpha_i)=
m_i^+(\alpha)-1$. Hence, the above expression evaluates to 
$(m_i^-(\alpha)-m_i^+(\alpha))v_\alpha=(\alpha,\alpha_i^\vee)v_\alpha$,
as required. 
\end{proof}

\begin{lem} \label{lem24} We have $e_i\circ f_j=f_j\circ e_i$ for 
all~$i\neq j$.
\end{lem}

\begin{proof} For $k\in I$, we have $e_i(f_j(u_k))=|(\alpha_j,\alpha_k^\vee)|
e_i(v_{-\alpha_j})=0$ since $i\neq j$ and $\alpha_i-\alpha_j\not\in\Phi$.
Similarly, we obtain that $f_j(e_i(u_k))=0$. Now let $\alpha\in\Phi$.

\smallskip
\noindent {\em Case 1}. Assume that $\alpha\in\{\pm\alpha_i,\pm\alpha_j\}$. 

If $\alpha=\alpha_i$, then $e_i(v_{\alpha_i})=0$ and $f_j(v_{\alpha_i})=0$ 
(since $\alpha_i-\alpha_j \not\in\Phi$). So both sides of the desired 
identity are equal to~$0$. The same happens for $\alpha=-\alpha_j$.

If $\alpha=\alpha_j$, then $f_j(v_{\alpha_j})=u_j$ and so $e_i(f_j
(v_{\alpha_j}))=|(\alpha_i,\alpha_j^\vee)|v_{\alpha_i}$. If $\alpha_i+
\alpha_j \not\in\Phi$, then $e_i(v_{\alpha_j})=0$ and so $f_j(e_i
(v_{\alpha_j}))=0$; furthermore, $(\alpha_i,\alpha_j)=0$ (see Lemma
~\ref{lem1}) and so $e_i(f_j(v_{\alpha_j}))=0$, as required. If $\alpha_i+ 
\alpha_j \in\Phi$, then 
\[f_j(e_i(v_{\alpha_j}))=m_i^-(\alpha_j)f_j(v_{\alpha_i+\alpha_j})=
m_i^-(\alpha_j)m_j^+(\alpha_i+\alpha_j)v_{\alpha_i}.\]
Now, since $\alpha_i-\alpha_j\not\in\Phi$, we have $m_i^-(\alpha_j)=
m_j^-(\alpha_i)=1$. Furthermore, we have $m_j^+(\alpha_i+\alpha_j)=
m_j^+(\alpha_i)-1=-(\alpha_i,\alpha_j^\vee)=|(\alpha_i,\alpha_j^\vee)|$
since $(\alpha_i,\alpha_j^\vee)\leq 0$. Hence, the desired identity
holds in this case as well.

If $\alpha=-\alpha_i$, then $e_i(v_{-\alpha_i})=u_i$ and so $f_j(e_i
(v_{-\alpha_i}))=|(\alpha_j,\alpha_i^\vee)|v_{-\alpha_j}$. Arguing
as in the previous case, we find the same result for $e_i(f_j
(v_{-\alpha_i}))$. 

\smallskip
\noindent {\em Case 2}. Assume that $\alpha\neq \pm\alpha_i$ and $\alpha\neq 
\pm\alpha_j$. If $\alpha+\alpha_i\not\in \Phi$ and $\alpha-\alpha_j\not
\in\Phi$, then $e_i(v_\alpha)=f_j(v_\alpha)=0$ and so both sides of the 
desired identity are~$0$. The same happens if $\alpha+\alpha_i-\alpha_j
\not\in\Phi$, as one checks immediately. (Note that $\alpha-\alpha_j\neq -
\alpha_i$ and $\alpha+\alpha_i\neq \alpha_j$.) Using Lemma~\ref{lem2}, it 
remains to consider the case where $\alpha+\alpha_i\in\Phi$, $\alpha-
\alpha_j \in\Phi$ and $\alpha+\alpha_i-\alpha_j\in\Phi$. Then we obtain
\begin{align*}
e_i(f_j(v_\alpha)) &=m_i^-(\alpha-\alpha_j)m_j^+(\alpha)v_{\alpha+\alpha_i-
\alpha_j},\\
f_j(e_i(v_\alpha)) &=m_j^+(\alpha+\alpha_i)m_i^-(\alpha)v_{\alpha+\alpha_i-
\alpha_j}.
\end{align*}
Hence, the desired identity holds by Lemma~\ref{lem3}. 
\end{proof}

\section{The Lie algebra generated by $e_i,f_i$} \label{sec3}

We consider the Lie algebra $\gl(M)$ consisting of all linear maps
$M\rightarrow M$ with the usual Lie bracket $[\varphi,\psi]=\varphi
\psi -\psi\varphi$. (We simply write $\varphi\psi$ instead
of $\varphi\circ \psi$ from now on.) For any subset $X\subseteq \gl(M)$,
we denote by  $\langle  X\rangle_{\text{Lie}}\subseteq \gl(M)$ the
Lie subalgebra generated by~$X$. Recall that this subalgebra is spanned
(as a vector space) by the set $\bigcup_{n\geq 1} X_n$, where the
subsets $X_n\subseteq \gl(M)$ are defined inductively by $X_1:=X$ and 
$X_{n+1}:=\{[y,z] \mid y\in X_k, z\in X_{n-k} \mbox{ where $1\leq k
\leq n$}\}$ for all $n\geq 1$. The elements in $X_n$ are called 
{\em Lie monomials} in $X$ (of level $n$).

We now define $\fg:=\langle e_i,f_i \mid i\in I\rangle_{\text{Lie}}
\subseteq \gl(M)$. Clearly, $\dim \fg<\infty$. By Lemma~\ref{lem23},
we have $h_i\in\fg$ for all $i\in I$. Let $\fh\subseteq \fg$ be the 
subspace spanned by $h_i$ ($i\in I$). By Remark~\ref{rem20}, this is an 
abelian subalgebra and the elements $h_i$ ($i\in I$) form a basis 
of~$\fh$. Our aim is to show that $\fg$ is a semisimple Lie algebra with 
Cartan subalgebra $\fh$ and root system isomorphic to $\Phi$. 

\begin{lem} \label{lem41} Let $\fn^+:=\langle e_i \mid i\in I
\rangle_{\operatorname{Lie}}$ and $\fn^-:=\langle f_i \mid i\in I
\rangle_{\operatorname{Lie}}$. Then, for a suitable ordering of the 
basis of $M$, all elements of $\fn^+$ are given by strictly upper triangular
matrices and all elements of $\fn^-$ are given by strictly lower 
triangular matrices. In particular, we have $\fg\subseteq \slm(M)$.
\end{lem}

\begin{proof} We write $\Phi^+=\{\beta_1,\ldots,\beta_N\}$ such that 
$\hgt(\beta_1)\leq \ldots \leq \hgt(\beta_N)$, where $\hgt(\beta_i)$ 
denotes the usual height of $\beta_i$ (with respect to the set of simple 
roots $\Pi$). Let $l=|I|$ and write $I=\{j_1,\ldots,j_l\}$. Then we order 
the basis elements of $M$ as
\[v_{\beta_N},\;\ldots,\;v_{\beta_1},\;\;u_{j_1},\ldots,u_{j_l},\;\;
v_{-\beta_1},\;\ldots,\; v_{-\beta_N}.\]
The definition in Section~\ref{sec2} immediately shows that the desired
statements hold for $e_i\in\fn^+$ and $f_i\in\fn^-$. Hence, they
also hold for all elements of $\fn^+$ and of $\fn^-$.
\end{proof}

\begin{lem} \label{prop4} Via $\fg\subseteq \gl(M)$, the 
vector space $M$ is a $\fg$-module. Then $M$ is an irreducible $\fg$-module. 
(Recall that $\Phi$ is assumed to be irreducible.)
\end{lem}

\begin{proof} Let $U\subseteq M$ be a $\fg$-submodule such that $U\neq\{0\}$.
We must show that $U=M$. Now every $v_\alpha\in M$ is a simultaneous 
eigenvector for all $h_i\in \fh$, with corresponding eigenvalue $(\alpha,
\alpha_i^\vee)$. Similarly, every $u_j \in M$ is a simultaneous eigenvector 
for all $h_i\in \fh$, with corresponding eigenvalue~$0$. Since $\Phi$ is a 
finite set, it is easy to see that there exists some $h_0\in \fh$ such 
that $h_0(v_\alpha) \neq 0$ and $h_0(v_\alpha) \neq h_0(v_\beta)$ for all 
$\alpha\neq \beta$ in $\Phi$. Now the restriction of $h_0$ to $U$ is also
diagonalisable. Hence, either $U\subseteq \langle u_i\mid i \in I 
\rangle_{\C}$ or there exists some $\alpha\in\Phi$ such that $v_\alpha
\in U$. 

Assume, if possible, that $U\subseteq \langle u_i \mid i \in I 
\rangle_{\C}$. Let $0\neq u\in U$ and write $u=\sum_{i\in I} x_i u_i$ 
where $x_i\in \C$. For $j\in I$, we have $e_j(u)=\sum_{i\in I} x_i
|(\alpha_j,\alpha_i^\vee)| v_{\alpha_j}$. Since $e_j(u)\in U$, we 
conclude that $\sum_{i \in I} x_i|(\alpha_j,\alpha_i^\vee)|=0$ 
for all $j\in I$. Hence, we have $\det(A')=0$ where $A':=\bigl(|(\alpha_j,
\alpha_i^\vee)| \bigr)_{i,j\in I}=4\,\mbox{id}-A$; here, $\mbox{id}$ is
the $I\times I$-identity matrix. But, using the classification of 
indecomposable Cartan matrices, one checks that $\det(A')=\det(A)\neq 0$, 
a contradiction. Hence, we are in the second case, that is, $v_\alpha\in U$
for some $\alpha \in\Phi$. We now proceed as follows to show that $U=M$.

(1) We claim that $u_{i_1}\in U$ for some $i_1\in I$. To see this, assume
first that $\alpha\in\Phi^+$. We can find a sequence $i_1,i_2,\ldots,i_h$ in
$I$ such that $\alpha=\alpha_{i_1}+\ldots+\alpha_{i_h}$ and $\alpha_{i_1}+
\ldots +\alpha_{i_l}\in\Phi$ for $1\leq l\leq h$ (see, e.g., \cite[\S 10.2, 
Corollary]{H}). But then the formulae in Section~\ref{sec2} show that 
$f_{i_1}\cdots f_{i_h}(v_\alpha)$ is a non-zero scalar multiple of 
$u_{i_1}$. Hence, since $v_\alpha\in U$, we also have $u_{i_1}\in U$. The 
argument is similar if $\alpha\in\Phi^-$.

(2) The formulae in Section~\ref{sec2} show that, for any $i,j\in I$, we 
have $e_if_i(u_j)=|(\alpha_i,\alpha_j^\vee)|u_i$. Hence, since $u_{i_1}
\in U$ and the Dynkin diagram of $\Phi$ is connected, we conclude that 
$u_i\in U$ for all $i\in I$.

(3) In order to complete the argument, it now suffices to show that
$v_\beta\in U$ for all $\beta\in \Phi$. To see this, assume first that 
$\beta\in\Phi^+$. As in (1), we can find a sequence $j_1,\ldots,j_k$ in 
$I$ such that $\beta=\alpha_{j_1}+\ldots+\alpha_{j_k}$ and $\alpha_{j_1}+
\ldots + \alpha_{j_m}\in\Phi$ for $1\leq m\leq k$. But then $e_{j_k}
\cdots e_{j_1}(u_{j_1})$ is seen to be a non-zero scalar multiple of 
$v_\beta$. Hence, since $u_{j_1}\in U$ by (2), we also have $v_\beta\in 
U$. The argument is similar if $\beta\in\Phi^-$.
\end{proof}

\begin{prop} \label{prop1} The Lie algebra $\fg$ is semisimple.
\end{prop}

\begin{proof} We have $\fg\subseteq \slm(M)$ (see Lemma~\ref{lem41})
and $M$ is an irreducible $\fg$-module (see Lemma~\ref{prop4}). 
Then it follows by a general argument that $\fg$ is semisimple; see, 
e.g., \cite[Exc.~12.4]{EW} (with hints in \cite[\S 20]{EW}) or  
\cite[\S 19.1]{H}. 
\end{proof}

In order to proceed, we need to introduce some further notation. Let $\fh^*
=\mbox{Hom}(\fh,\C)$ be the dual space. For $\alpha\in \Phi$, we define
$\dalp\in\fh^*$ by $\dalp(h_j):=(\alpha,\alpha_j^\vee)$ for all $j\in I$. 
Let $\dPhi:=\{\dalp\mid \alpha \in\Phi\}$. Since the Cartan matrix $A$ is 
invertible, the elements $\dalp_i$ ($i\in I$) form a basis of~$\fh^*$. For 
$\lambda \in\fh^*$, we define
\[\fg_\lambda:=\{x\in\fg\mid [h,x]=\lambda(h)x \mbox{ for all $h\in 
\fh$}\}.\]
If $\lambda=\dalp$, we also write $\fg_\alpha$ instead of $\fg_{\dalp}$.
Now Lemma~\ref{lem22} shows that $[h_j,e_i]=(\alpha_i,\alpha_j^\vee)e_i=
\dalp_i(h_j)e_i$ for all $j\in I$. Hence, $e_i\in \fg_{\dalp_i}=
\fg_{\alpha_i}$; similarly, $f_i\in \fg_{-\dalp_i}=\fg_{-\alpha_i}$. 
Also note that the map $\alpha\mapsto\dalp$ is linear in $\alpha$.

\begin{lem} \label{lem40} We have $\fh\subseteq \fg_0$ and $\fn^\pm 
\subseteq\sum_{\lambda\in Q_\pm} \fg_\lambda$, where we set  
\begin{center}
$Q_+:=\bigl\{\sum_{i\in I} n_i\dalp_i\in \fh^*\mid n_i\in
\Z_{\geq 0} \mbox{ for all $i\in I$}\bigr\}\setminus \{0\}$,
\end{center}
\begin{center}
$Q_-:=\bigl\{\sum_{i\in I} n_i\dalp_i\in \fh^*\mid n_i\in
\Z_{\leq 0} \mbox{ for all $i\in I$}\bigr\}\setminus \{0\}.$
\end{center}
\end{lem}

\begin{proof} This is analogous to step~(5) of the proof of 
\cite[Theorem~18.2]{H}. Since $\fh$ is abelian, $\fh\subseteq \fg_0$. Now 
let $x\in \fn^+$ be a Lie monomial in $\{e_i\mid i\in I\}$ of level~$n$. We 
show by induction on $n$ that $x\in \fg_\lambda$ for some $\lambda\in Q_+$.
If $n=1$, then $x=e_i$ for some $i$ and we already noted that $e_i\in 
\fg_{\dalp_i}$. If $n\geq 2$, then $x=[y, z]$ where $y,z$ are Lie 
monomials of level $k$ and $n-k$, respectively. By induction, $y\in 
\fg_\mu$ and $z\in\fg_\nu$ where $\lambda, \nu \in Q_+$. A computation
using the Jacobi identity shows that $[h,x]=[h,[y,z]]=(\mu+ \nu)(h)x$ for 
all $h\in\fh$. Hence, $x\in \fg_\lambda$ where $\lambda:=\mu+\nu\in Q_+$. 
The argument for $\fn^-$ is completely analogous.
\end{proof}

\begin{lem} \label{lem42} We have a direct sum decomposition $\fg=
\fn^-\oplus \fh\oplus \fn^+$. 
\end{lem}

\begin{proof} This is analogous to steps~(6), (7), (8) of the proof of 
\cite[Theorem~18.2]{H}. The crucial property to show is that $[e_i,\fn^-]
\subseteq \fn^-+\fh$ for all $i\in I$. This is done as follows. 
Let $x\in\fn^-$ be a Lie monomial in $\{f_j\mid j\in I\}$ of level~$n$.
If $n=1$, then $x=f_j$ for some~$j$ and so $[e_i,f_j]\in\fh$ by 
Lemmas~\ref{lem23} and~\ref{lem24}. If $n\geq 2$, then $x=[y,z]$ where
$y,z$ are Lie monomials of level $k$ and $n-k$, respectively. Using 
induction and the Jacobi identity, we obtain $[e_i,x]=[y,[e_i,z]]+
[z,[e_i,y]]\in [y,\fh]+[y,\fn^-]+[z,\fh]+[z,\fn^-]$. Clearly, $[y,\fn^-]
\subseteq \fn^-$ and $[z,\fn^-]\subseteq \fn^-$; furthermore, by
Lemma~\ref{lem40} (and its proof), we have $[y,\fh]\subseteq \C y$ and 
$[z,\fh]\subseteq \C z$. Hence, $[e_i,x]\in\fn^-$, as required. By a 
completely analogous argument, one shows that $[f_i,\fn^+] \subseteq 
\fn^++\fh$ for all $i\in I$. Hence, setting $V:=\fn^-+\fh+\fn^+\subseteq 
\fg$, we see that $[e_i,V]\subseteq V$ and $[f_i,V]\subseteq V$ for all 
$i\in I$. By a further induction on the level of Lie monomials, this 
implies that $[V,V]\subseteq V$ and so $V$ is a subalgebra of $L$. Since 
$e_i,f_i\in V$, we conclude that $V=\fg$. Finally, directness of the sum 
$V=\fn^-+\fh+ \fn^+$ follows from Lemma~\ref{lem41} and the fact that all 
elements of $\fh$ are given by diagonal matrices.
\end{proof}

\begin{thm} \label{fincor} Recall that $\Phi$ is assumed to be irreducible. 
Then the Lie algebra $\fg=\langle e_i,f_j\mid i,j\in I 
\rangle_{\operatorname{Lie}} \subseteq \gl(M)$ is simple, $\fh$ is a Cartan 
subalgebra of $\fg$ and $\dPhi\cong \Phi$ is the root system of $\fg$
with respect to $\fh$. In particular, we have a direct sum decomposition 
$\fg =\fh \oplus \bigoplus_{\alpha \in\Phi} \fg_{\alpha}$ where $\dim 
\fg_{\alpha}=1$ for $\alpha\in\Phi$.
\end{thm}

\begin{proof} This is now a matter of putting the above pieces together.
By Proposition~\ref{prop1}, $\fg$ is semisimple. By
Lemma~\ref{lem40}, $\fh\subseteq \fg_0$ and $\fn^\pm 
\subseteq\bigoplus_{\lambda\in Q_\pm} \fg_\lambda$. Using also 
Lemma~\ref{lem42}, we deduce that $\fh=\fg_0$ and $\fn^\pm=\sum_{\lambda 
\in Q_\pm} \fg_\lambda$, which in turn implies that $\fh$ is a 
Cartan subalgebra of~$\fg$. 

Let $\Phi'$ be the root system of $\fg$ with 
respect to $\fh$. Then $\Phi'\subseteq Q_+\cup Q_-$. Since $\dalp_i\in
\Phi'$ for all $i\in I$, we deduce that $\dPi:=\{\dalp_i\mid i\in I\}$ is 
a set of simple roots for~$\Phi'$. By Lemmas~\ref{lem22} and \ref{lem23}, 
we have $[e_i,f_i]=h_i$ and $[h_j,e_i]=(\alpha_i,\alpha_j^\vee)e_i$ for 
all $i,j$. Hence, the Cartan matrix $A$ is also the Cartan matrix of 
$\Phi'$ with respect to $\dPi$. But every root system is uniquely determined
by a set of simple roots and the corresponding Cartan matrix. Hence, 
$\Phi'=\dPhi$. Finally, $\fg$ is simple since $\dPhi\cong \Phi$ is 
irreducible. The statements about the direct sum decomposition of $\fg$ 
are classical facts about semisimple Lie algebras; see, e.g., 
\cite[Chap.~10]{EW}, \cite[\S 8.4]{H}.
\end{proof}

\begin{rem} \label{coroots} For any $\alpha\in\Phi$ we define a linear 
map $h_\alpha\colon M\rightarrow M$ by $h_\alpha(u_j):=0$ for $j\in I$ 
and $h_\alpha(v_\beta):=(\beta,\alpha^\vee)v_\beta$ for $\beta\in \Phi$. 
(If $\alpha=\alpha_i$ for some $i\in I$, then this agrees with the 
definition of $h_i$.) Now, we can certainly write $\alpha^\vee=
\sum_{i \in I} x_i \alpha_i^\vee$ where $x_i\in \C$. Then $h_\alpha
(v_{\beta})=(\beta,\alpha^\vee)v_\beta=\sum_{i\in I} x_i(\beta,
\alpha_i^\vee)v_\beta=\sum_{i\in I}x_i h_i(v_{\beta})$ for all $\beta
\in\Phi$. We conclude that $h_\alpha=\sum_{i \in I}x_ih_i\in \fh$ 
and $\dbet(h_\alpha)=(\beta,\alpha^\vee)\in \Z$ for all $\beta \in\Phi$.
Thus, the elements $\{h_\alpha\mid \alpha\in\Phi\}\subseteq \fh$ are 
the ``co-roots'' of $\fg$.
\end{rem}

\begin{rem} \label{rem200a} Let $\omega\colon M\rightarrow M$ be as
in Remark~\ref{lem20}. Then conjugation with $\omega$ preserves $\fg$ and
so we obtain a Lie algebra automorphism $\tilde{\omega}\colon \fg
\rightarrow \fg$ such that 
\[\tilde{\omega}^2=\mbox{id}_{\fg}, \qquad  \tilde{\omega}(e_i)=f_i
\quad \mbox{and}\quad \tilde{\omega}(h_i)=-h_i\quad \mbox{for all $i\in I$}.\]
Consequently, we have $\tilde{\omega}(\fh)=\fh$ and $\tilde{\omega}
(\fg_\alpha)=\fg_{-\alpha}$ for all $\alpha\in\Phi$. 
\end{rem}

\begin{rem} \label{remadj} Let $\alpha_0\in \Phi$ be the unique root of 
maximal height (with respect to~$\Pi$). Then the formulae in 
Section~\ref{sec2} show that $v_{\alpha_0}\in M$ is a primitive vector, 
with corresponding weight~$\dalp_0$. Hence, by the general theory of 
highest weight modules (see \cite[Chap.~VI]{H}, \cite[Chap.~VII]{S}), 
the simple $\fg$-module $M$ is isomorphic to $\fg$ (viewed as a $\fg$-module 
via the adjoint representation). If one does not want to use these general
results, then one has to show directly that $\fg$ admits a basis on
which $e_i,f_i$ act via the formulae in Section~\ref{sec2}, following
the argument in \cite{L1}, \cite{L2}.
\end{rem}

\begin{rema} {\bf Chevalley groups.} \label{cheva} Following Lusztig 
\cite{L4}, \cite[\S 2]{L5}, we obtain the Chevalley groups corresponding
to $\fg\subseteq \gl(M)$ as follows. By Lemma~\ref{lem41}, every $x\in 
\fn^\pm$ is a nilpotent linear map. Hence, we can define $\exp(x)\in 
\GL(M)$; note that $\exp(x) \exp(-x)=\mbox{id}_M$. In particular, we can 
define $x_i(t):=\exp(t e_i)\in \GL(M)$ and $y_i(t):=\exp(t f_i)\in \GL(M)$ 
for all $i\in I$, $t\in\C$. Explicitly, we have: 
\begin{gather*}
x_i(t)(u_j) = u_j+|(\alpha_i,\alpha_j^\vee)|tv_{\alpha_i},\qquad
x_i(t)(v_{-\alpha_i}) =v_{-\alpha_i}+tu_i+t^2v_{\alpha_i},\\
x_i(t)(v_{\alpha_i}) =v_{\alpha_i}, \qquad x_i(t)(v_{\alpha}) =
\sum_{k\geq 0,\, \alpha+k\alpha_i\in\Phi} \binom{k+m_i^-(\alpha)-1}{k}
t^kv_{\alpha+ k\alpha_i},\\
y_i(t)(u_j) = u_j+|(\alpha_i,\alpha_j^\vee)|tv_{-\alpha_i},\qquad
y_i(t)(v_{\alpha_i}) =v_{\alpha_i}+tu_i+t^2v_{-\alpha_i},\\
y_i(t)(v_{-\alpha_i}) =v_{-\alpha_i}, \qquad y_i(t)(v_{\alpha}) =
\sum_{k\geq 0,\,\alpha-k\alpha_i\in\Phi} \binom{k+m_i^+(\alpha)-1}{k}
t^kv_{\alpha-k\alpha_i},
\end{gather*}
where $j\in I$ and $\alpha\in \Phi$, $\alpha\neq \pm \alpha_i$. (Compare
with the formulae in \cite[\S 4.3]{Ca1}, \cite[p.~24]{Ch}.) Now let $R$ be 
any commutative ring with $1$ and $\bar{M}$ be a free $R$-module with a 
basis $\{\bar{u}_i \mid i\in I\} \cup \{\bar{v}_\alpha \mid \alpha\in
\Phi\}$. Using a specialisation argument as in \cite{Ch}, we can 
then define $\bar{x}_i(t)\in \GL(\bar{M})$ and $\bar{y}_i(t)\in 
\GL(\bar{M})$ for all $i\in I$, $t\in R$. (See also \cite[\S 4.4]{Ca1}.) 
The corresponding Chevalley group (of adjoint type) is defined as 
\[G_R:=\langle \bar{x}_i(t), \bar{y}_i(t)\mid i \in I, t\in R \rangle
\subseteq \GL(\bar{M}).\]
In this way, we obtain a canonical procedure $R\leadsto G_R$, 
which does not involve the choice of certain signs as in Chevalley's
original approach \cite{Ch}. 
\end{rema}

\section{Chevalley bases} \label{sec4}

Let $\fg\subseteq \gl(M)$ be as in the previous section. In order to 
establish further structural properties of the corresponding groups $G_R$ 
(e.g., Chevalley's commutator relations), one needs to define ``integral'' 
elements $\be_\alpha \in \fg_{\alpha}$ for all $\alpha\in \Phi$. 
For this purpose, recall from Remark~\ref{remadj} that we have an 
isomorphism of $\fg$-modules $M\cong \fg$, where $\fg$ is viewed as a 
$\fg$-module via the adjoint representation. Since $M$ is irreducible, 
such an isomorphism is unique up to multiplication by a scalar. The first
step now is to see how we can fix a specific isomorphism $M\cong \fg$. 

Since the Dynkin diagram of $\Phi$ has no loops, there are exactly two 
functions $\epsilon\colon I\rightarrow \{\pm 1\}$ such that $\epsilon(i)
=-\epsilon(j)$ whenever $a_{ij}\neq 0$ for $i\neq j$ in $I$; if $\epsilon$ 
is one of these two functions, then the other one is $-\epsilon$. The 
following result is due to Lusztig (unpublished); essentially the same
statement appears in Rietsch \cite[4.1]{KR}. (I~thank Lusztig for
pointing out this reference to me.)

\begin{lem} \label{remadj0} Let us fix a function $\epsilon\colon I
\rightarrow \{\pm 1\}$ as above. Then there is a unique $\fg$-module
isomorphism $\varphi \colon M\rightarrow \fg$ such that, for all
$i\in I$, we have 
\[\varphi(v_{\alpha_i})=\epsilon(i)e_i, \quad \varphi(v_{-\alpha_i})=-
\epsilon(i)f_i,\quad\varphi(u_i)=-\epsilon(i) h_i.\]
\end{lem}

\begin{proof} We know that there exists some $\fg$-module isomorphism
$\varphi\colon M\rightarrow \fg$. Let $i\in I$. Since $[h_j,\varphi
(v_{\alpha_i})]=\varphi(h_j(v_{\alpha_i}))=(\alpha_i,\alpha_j^\vee)
\varphi(v_{\alpha_i})$ for all $j\in I$, it is clear that $\varphi
(v_{\alpha_i})\in\fg_{\alpha_i}$ and, hence, $\varphi(v_{\alpha_i})=
c_ie_i$ where $0\neq c_i\in \C$. This then implies that $\varphi(u_i)=
\varphi(f_i(v_{\alpha_i}))=[f_i,\varphi(v_{\alpha_i})]=c_i[f_i,e_i]=
-c_ih_i$. Similarly, we have $\varphi(v_{-\alpha_i})\in \fg_{-\alpha_i}$
and $e_i(v_{-\alpha_i})=u_i$, which implies that $\varphi(v_{-\alpha_i})=
-c_if_i$. 

Now assume that $i\neq j$ in $I$ are such that $a_{ij}\neq 0$; note that
$a_{ij}<0$. Then $\varphi(e_i(u_j))=|a_{ji}|\varphi(v_{\alpha_i})=-a_{ji}
c_ie_i$ and $[e_i,\varphi(u_j)]=-c_j[e_i,h_j]=c_ja_{ji}e_i$. Since these two 
expressions are equal and non-zero, we must have $c_i=-c_j$. We conclude
that all $c_i$ have the same value up to sign; furthermore, $c_i,c_j$ have 
opposite signs whenever $i\neq j$ in $I$ are such that $a_{ij}\neq 0$. 
Consequently, there is some $0\neq c \in \C$ such that $c_i=c\epsilon(i)$
for all $i\in I$. It remains to replace $\varphi$ by $c^{-1}\varphi$.
\end{proof}

\begin{defn} \label{chbase0} Let $\epsilon$ and $\varphi\colon M
\rightarrow \fg\subseteq \gl(M)$ be as in Lemma~\ref{remadj0}. Then we set 
\[\be_\alpha^\epsilon:=\varphi(v_\alpha) \qquad \mbox{for all 
$\alpha \in\Phi$}.\]
In particular, $\be_{\alpha_i}^\epsilon=\epsilon(i)e_i$ and 
$\be_{-\alpha_i}^\epsilon=-\epsilon(i)f_i$ for all $i\in I$. Since each 
$v_\alpha$ belongs to the $\dalp$-weight space of $M$, it is clear that 
$\be_\alpha^\epsilon\in \fg_\alpha$. Hence, 
\[B^\epsilon:=\{h_i\mid i\in I\}\cup \{\be_\alpha^\epsilon\mid 
\alpha\in\Phi\}\:\mbox{ is a basis of $\fg$}.\]
If $\alpha,\beta,\alpha+\beta\in \Phi$, then we write as usual 
$[\be_\alpha^\epsilon,\be_\beta^\epsilon]=N_{\alpha,\beta}^\epsilon 
\be_{\alpha+\beta}^\epsilon$ where $N_{\alpha,\beta}^\epsilon\in\C$. 

Note that, if we replace $\epsilon$ by $-\epsilon$, then $\varphi$ 
is replaced by $-\varphi$. Hence, $\be_\alpha^{-\epsilon}=
-\be_\alpha^\epsilon$ and $N_{\alpha,\beta}^{-\epsilon}=-N_{\alpha,
\beta}^\epsilon$ for all $\alpha,\beta \in\Phi$ such that $\alpha+
\beta\in\Phi$. Thus, the passage from $B^\epsilon$ to $B^{-\epsilon}$
is determined by a very simple and explicit rule.
\end{defn}

In order to describe the elements $\be_\alpha^\epsilon$ more explicitly,
we need one further ingredient. In the set-up of Section~\ref{sec1}, let 
$W\subseteq \GL(E)$ be the Weyl group of $\Phi$; we have $W=\langle s_i
\mid i \in I \rangle$ where $s_i\colon E \rightarrow E$ is defined by 
$s_i(e)=e-(e, \alpha_i^\vee) \alpha_i$ for $e \in E$. It is well-known 
that the generators $s_i\in W$ can be lifted to automorphisms of~$\fg$. 
Indeed, following \cite[Chap.~VIII, \S 2, no.~2, formule~(1)]{B}, we set 
for any $i\in I$:
\[ n_i(t):=\exp(te_i)\exp(-t^{-1}f_i)\exp(te_i)\colon M\rightarrow M 
\qquad \mbox{where $0\neq t\in \C$}.\]
The maps $n_i(t)\colon M\rightarrow M$ are compatible with $\varphi
\colon M\rightarrow \fg$ by the following rule.

\begin{lem} \label{remadj1} We have $n_i(t)\varphi(m)n_i(t)^{-1}=\varphi
\bigl(n_i(t)(m)\bigr)$ for all $m\in M$, $i\in I$ and $0\neq t\in\C$. 
(Here, conjugation with $n_i(t)$ on the left hand side takes place inside 
$\gl(M)$.) In particular, we have $n_i(t)\fg n_i(t)^{-1}\subseteq \fg$.
\end{lem}

\begin{proof} Let $x\in\fg$ and assume that $x$ is a nilpotent linear map; 
thus, we can form $\exp(x)\in \GL(M)$. Then $\mbox{ad}(x)\colon\fg 
\rightarrow \fg$ is a nilpotent derivation and we have $\exp(x)\varphi(m)
\exp(x)^{-1}=\exp(\mbox{ad}(x))(\varphi(m))$; see, e.g., \cite[4.3.1, 
4.5.1]{Ca1}. Now, we have $\mbox{ad}(x)(\varphi(m))=[x,\varphi(m)]=
\varphi(x(m))$ and so $\mbox{ad}(x)^k(\varphi(m))=\varphi(x^k(m))$ for all
$k\geq 0$. This immediately implies that $\exp(\mbox{ad}(x))(\varphi(m))=
\varphi(\exp(x)(m))$ and so 
\[\exp(x)\varphi(m)\exp(x)^{-1}=\varphi(\exp(x)(m)).\]
Applying this rule with $x=te_i$ and $x=-t^{-1}f_i$ yields the desired 
statement.
\end{proof}

\begin{lem} \label{actni} Let $i\in I$. Then we have for any $j\in I$ 
and $\alpha\in \Phi$:
\begin{align*}
n_i(t)(u_j)&=u_j-|(\alpha_i,\alpha_j^\vee)|u_i,\\
n_i(t)(v_\alpha)&=\left\{\begin{array}{cl} 
t^{-2}v_{-\alpha_i} & \mbox{ if $\alpha=\alpha_i$},\\ 
t^2v_{\alpha_i} & \mbox{ if $\alpha=-\alpha_i$},\\ 
-(-1)^{m_i^-(\alpha)}t^{-(\alpha,\alpha_i^\vee)}v_{s_i(\alpha)}&
\mbox{ otherwise}. \end{array}\right.
\end{align*}
Consequently, $n_i(t)^2(u_j)=u_j$, $n_i(t)^2(v_\alpha)=
(-1)^{(\alpha,\alpha_i^\vee)}v_\alpha$ and so $n_i(t)^4=
\operatorname{id}_M$. 
\end{lem}

(Note that, with respect to the adjoint representation of $\fg$, analogous 
formulae are only known up to signs, see \cite[p.~36]{Ch} and also 
\cite[Prop.~6.4.2]{Ca1}.)

\begin{proof} Using the formulae in \ref{cheva} for the action of $x_i(t)$ 
and $y_i(-t^{-1})$ on $M$, it is straightforward to check the formulae for 
$n_i(t)(u_j)$ and $n_i(t)(v_{\pm\alpha_i})$. In particular, we see that 
the subspace $M_0:=\fh+\C v_{\alpha_i}+\C v_{-\alpha_i}\subseteq M$ is 
invariant under $n_i(t)$. Now define an equivalence relation~$\sim$ on $\Phi
\setminus\{\pm \alpha_i\}$ by $\alpha\sim \alpha'$ if $\alpha'-\alpha=
m\alpha_i$ for some $m\in\Z$. Then we have a direct sum decomposition 
$M=M_0\oplus \bigoplus_{\cO} M_\cO$ where $\cO$ runs over the equivalence 
classes of $\Phi\setminus \{\pm \alpha_i\}$ (with respect to $\sim$) and 
$M_{\cO}:=\langle v_\alpha\mid \alpha\in \cO\rangle_\C$. By the definitions 
of $\exp(te_i)$ and $\exp(-t^{-1}f_i)$, each of the direct summands $M_\cO$ 
is invariant under~$n_i(t)$. So we can verify the desired identity term by 
term in this decomposition. Now let $\alpha\in\Phi$, $\alpha\neq\pm\alpha_i$.
Let $\cO$ be the equivalence class of $\alpha$. Then $\cO$ is the
$\alpha_i$-string through $\alpha$ (see \ref{strgs}) and so we can write 
$\cO=\{\beta,\beta+\alpha_i,\ldots,\beta+p\alpha_i\}$ where $\beta\in\cO$
and $p=-(\beta, \alpha_i^\vee) \in \{0,1,2,3\}$; thus, $\alpha=\beta+
k\alpha_i$ where $0\leq k\leq p$. Now, by the definition of $e_i,f_i$, we 
have:
\begin{align*}
e_i\colon & v_\beta  \mapsto  v_{\beta+\alpha_i}, \;\; v_{\beta+\alpha_i}
\mapsto 2v_{\beta+2\alpha_i},\;\; \ldots,\;\;  v_{\beta+(p-1)\alpha_i}
\mapsto p v_{\beta+p\alpha_i},\;\, v_{\beta+p\alpha_i} \mapsto 0,\\ 
f_i\colon & v_\beta \mapsto 0, \;\; v_{\beta+\alpha_i} \mapsto pv_{\beta},
\;\;\ldots, \;\; v_{\beta+2\alpha_i}\mapsto (p{-}1)v_{\beta+\alpha_i},
\;\; v_{\beta+p\alpha_i} \mapsto v_{\beta+(p-1)\alpha_i}.
\end{align*}
Hence, we can explicitly compute the matrix of $x_i(t)$ and $y_i(-t^{-1})$
on $M_\cO$. For $p=0,1,2,3$ the resulting matrix of $n_i(t)$ on $M_\cO$
is given by 
\[(1),\quad\left(\begin{array}{cc} 0 & -t^{-1} \\ t & 0\end{array}\right),
\quad \left(\begin{array}{ccc} 0 & 0  & t^{-2}\\0 & -1 & 0 \\ t^2&0 & 0
\end{array}\right), \quad\left(\begin{array}{cccc} 0&0 &0 &-t^{-3} \\ 
0&0 &t^{-1} & 0 \\ 0 &-t &0 &0 \\t^3 &0 &0 &0 \end{array}\right),\]
respectively. Now, if $\alpha=\beta+k\alpha_i$, then 
$(\alpha,\alpha_i^\vee)=2k-p$ and $s_i(\alpha)=\alpha-(2k-p)\alpha_i=
\beta+p \alpha_i-k\alpha_i$; furthermore, $m_i^-(\alpha)=k+1$. Then the
desired formulae for $n_i(t)(v_\alpha)$ can be simply read off the above
matrices. 
\end{proof}

\begin{lem} \label{acthal} We have $n_i(t)h_\alpha=h_{s_i(\alpha)}n_i(t)$
for all $i\in I$, $0\neq t\in\C$ and $\alpha\in\Phi$. (See
Remark~\ref{coroots} for the definition of $h_\alpha$.)
\end{lem}

\begin{proof} This is a straightforward verification using the formulae
in Lemma~\ref{actni} and the fact that $(s_i(\alpha))^\vee=
s_i(\alpha^\vee)$; cf.\ \cite[Chap.~VIII, \S 2, Lemme~1]{B}.
\end{proof}

\begin{cor} \label{intea} Let $\alpha\in\Phi$. Then $\be_\alpha^\epsilon
\colon M\rightarrow M$ is nilpotent. The matrix of $\be_\alpha^\epsilon$ 
with respect to the given basis of $M$ has entries in $\Z$ and these
entries are relatively prime. For every integer $k\geq 0$, the matrix of 
$\frac{1}{k!} (\be_\alpha^\epsilon)^k$ also has entries in $\Z$. 
\end{cor}

\begin{proof} We can write $\alpha=w(\alpha_i)$ for some $w\in W$ and 
$i\in I$; furthermore, $w=s_{i_1} \cdots s_{i_l}$ where $l\geq 0$, $i_1,
\ldots,i_l\in I$. Then we set $\eta:=n_{i_1}(1)\cdots n_{i_l}(1)\colon
M \rightarrow M$. By Lemma~\ref{actni}, each $n_j(1)$ is represented by 
a matrix with entries in $\Z$; furthermore, $n_j(1)^4=\mbox{id}_M$ and 
so $\det(n_j(1))=\pm 1$. Consequently, $\eta$ is also represented by
a matrix with entries in $\Z$ and we have $\det(\eta)=\pm 1$. Using 
Lemma~\ref{actni}, we obtain that $v_\alpha= v_{w(\alpha_i)}=\delta\eta 
(v_{\alpha_i})$ where $\delta=\pm 1$. This yields the formula:
\begin{equation*}
\be_\alpha^\epsilon=\varphi(v_\alpha)=\delta\varphi(\eta(v_{\alpha_i}))=
\delta\eta\varphi(v_{\alpha_i})\eta^{-1}=\delta\epsilon(i)\eta e_i\eta^{-1}
=\pm \eta e_i\eta^{-1}, \tag{$*$}
\end{equation*}
where the third equality holds by Lemma~\ref{remadj1} and the fourth 
equality by Lemma~\ref{remadj0}. Since $e_i$ is nilpotent, ($*$) shows
that $\be_\alpha^\epsilon$ is also nilpotent. Now consider the integer 
matrix representing $e_i$; since $e_i(v_{-\alpha_i})=u_i$, at least one 
entry of this matrix is~$1$. Using ($*$), we conclude that the 
entries of the matrix of~$\be_\alpha^\epsilon$ are in $\Z$ and they are 
relatively prime. Similarly, since the matrix of $\frac{1}{k!}e_i^k$ has 
entries in $\Z$ (this is implicit in the formulae in \ref{cheva}),
the same is true for $\frac{1}{k!}(\be_\alpha^\epsilon)^k$ by ($*$).
\end{proof}

Note that the elements $\be_\alpha^\epsilon$ can be explicitly computed
using ($*$) in the above proof. Furthermore, since the map 
$\be_\alpha^\epsilon\colon M\rightarrow M$ is nilpotent, we can define 
\[ x_\alpha^\epsilon(t):=\exp(t\be_\alpha^\epsilon)\in\GL(M) \qquad 
\mbox{for any $t\in\C$}.\]
Now let $R$ be a commutative ring with $1$, as in \ref{cheva}. Since 
$\frac{1}{k!}(\be_\alpha^\epsilon)^k$ is represented by an integer matrix 
for any integer $k\geq 0$, we can again apply a specialisation argument 
and also obtain elements $\bar{x}_\alpha^\epsilon(t)\in \GL(\bar{M})$ for 
all $t\in R$. Using ($*$) once more, one easily sees that 
$\bar{x}_\alpha^\epsilon(t)\in G_R$. Thus, we have 
\[G_R=\langle \bar{x}_\alpha^\epsilon(t)\mid \alpha\in\Phi,t\in R\rangle
\subseteq \GL(\bar{M}).\]
Also note that $\bar{x}_\alpha^{-\epsilon}(t)=\bar{x}_\alpha^\epsilon(-t)$ 
for all $\alpha\in\Phi$ and $t\in R$. 

\begin{thm}[Cf.\ Chevalley \protect{\cite[Th\'eor\`eme 1]{Ch}}] 
\label{chbase} We have the 
following relations.

{\rm (a)} $\tilde{\omega}(\be_\alpha^\epsilon)=-\be_{-\alpha}^\epsilon$ 
for all $\alpha\in \Phi$, with $\tilde{\omega} \colon\fg \rightarrow\fg$ 
as in Remark~\ref{rem200a}. 

{\rm (b)} $[\be_\alpha^\epsilon,\be_{-\alpha}^\epsilon]=
(-1)^{\hgt(\alpha)}h_\alpha$ for all $\alpha\in \Phi$, where $\hgt(\alpha)$
denotes the height of $\alpha$. 

{\rm (c)} $N_{\alpha, \beta}^\epsilon=\pm m_\alpha^-(\beta)$ if 
$\alpha,\beta,\alpha+\beta\in\Phi$.
\end{thm}

\begin{proof} (a) Since $\tilde{\omega}^2=\mbox{id}_\fg$, it is enough to 
prove this for $\alpha\in \Phi^+$. We use induction on $\hgt(\alpha)$.
If $\hgt(\alpha)=1$, then $\alpha=\alpha_i$ for some $i \in I$. So the
assertion holds since $\be_{\alpha_i}^\epsilon=\epsilon(i)e_i$, 
$\be_{-\alpha_i}^\epsilon=-\epsilon(i)f_i$ and $\tilde{\omega}(e_i)=f_i$.
Now assume that $\hgt(\alpha)>1$. There is some $i\in I$ such that 
$(\alpha,\alpha_i^\vee)>0$ and $\beta:=\alpha-\alpha_i\in \Phi^+$. Then 
$\tilde{\omega}(\be_\beta^\epsilon)=-\be_{-\beta}^\epsilon$ by induction. 
Since $e_i(v_\beta)=m_i^-(\beta)v_\alpha$ and $f_i(v_{-\beta})=
m_i^+(-\beta)v_{-\alpha}$, we obtain that 
\begin{align*}
m_i^-(\beta)&\tilde{\omega}(\be_\alpha^\epsilon)=\tilde{\omega}\bigl(
\varphi(e_i(v_\beta))\bigr)=\tilde{\omega}([e_i,\varphi(v_\beta)])
=[\tilde{\omega}(e_i),\tilde{\omega}(\be_\beta^\epsilon)]=[f_i,
-\be_{-\beta}^\epsilon]\\ &=-[f_i,\varphi(v_{-\beta})]=-\varphi(f_i
(v_{-\beta}))=-m_i^+(-\beta)\varphi(v_{-\alpha})=-m_i^+(-\beta)
\be_{-\alpha}^\epsilon.
\end{align*}
It remains to use the fact that $m_i^-(\beta)=m_i^+(-\beta)$.

(b) Since $h_{-\alpha}=-h_\alpha$, it is enough to prove this for $\alpha
\in \Phi^+$. Again, we use induction on $\hgt(\alpha)$. If $\hgt(\alpha)=1$, 
then $\alpha=\alpha_i$ for some $i \in I$. So the desired formula holds
since $\be_{\alpha_i}^\epsilon=\epsilon(i)e_i$, $\be_{-\alpha_i}^\epsilon=
-\epsilon(i)f_i$ and $[e_i,f_i]=h_i=h_{\alpha_i}$. Now assume that
$\hgt(\alpha)>1$. As above, there is some $i\in I$ such that $(\alpha,
\alpha_i^\vee)>0$. Then $s_i(\alpha)\in \Phi^+$ and $\hgt(s_i(\alpha))=
\hgt(\alpha)-(\alpha,\alpha_i^\vee)$. So, by induction, 
$[\be_{s_i(\alpha)}^\epsilon,\be_{-s_i(\alpha)}^\epsilon]=
(-1)^{\hgt(s_i(\alpha))}h_{s_i(\alpha)}$. Since $\alpha\neq \pm \alpha_i$,
Lemmas~\ref{remadj1} and~\ref{actni} show that 
\[\be_{\pm \alpha}^\epsilon =n_i(1)^{-1}\varphi(n_i(1)(v_{\pm \alpha}))
n_i(1)=-(-1)^{m_i^-(\pm \alpha)}n_i(1)^{-1}\be_{\pm s_i(\alpha)}^\epsilon 
n_i(1).\]
Now note that $m_i^-(-\alpha)=m_i^+(\alpha)$. So, by \ref{strgs}, we 
have $(\alpha,\alpha_i^\vee)=m_i^-(\alpha)-m_i^-(-\alpha)$ and, hence, 
$\hgt(\alpha)\equiv m_i^-(\alpha)+m_i^-(-\alpha)+\hgt(s_i(\alpha))\bmod 2$.
This yields that 
\[[\be_\alpha^\epsilon,\be_{-\alpha}^\epsilon]=(-1)^{\hgt(\alpha)}
n_i(1)^{-1}h_{s_i(\alpha)}n_i(1)=(-1)^{\hgt(\alpha)}h_\alpha,\]
where the last equality holds by Lemma~\ref{acthal}.

(c) As in the proof of Lemma~\ref{intea}, write $\alpha=w(\alpha_i)$ where 
$w\in W$, $i\in I$; furthermore, $w=s_{i_1} \cdots s_{i_l}$ where 
$l\geq 0$, $i_1, \ldots,i_l\in I$. Then $\be_\alpha^\epsilon=\pm
\eta e_i\eta^{-1}$ where $\eta:=n_{i_1}(1)\cdots n_{i_l}(1)\colon
M \rightarrow M$. Now set $\gamma:=w^{-1}(\beta)\in\Phi$. 
By an analogous argument we find that $v_\beta=v_{w(\gamma)}=\pm 
\eta(v_\gamma)$, hence $\be_\beta^\epsilon=\varphi(v_\beta)=\pm \eta
\be_\gamma^\epsilon\eta^{-1}$ and so 
\[ [\be_\alpha^\epsilon,\be_\beta^\epsilon]=\pm [\eta e_i\eta^{-1},\eta
\be_\gamma^\epsilon\eta^{-1}]=\pm \eta[e_i,\be_\gamma^\epsilon]\eta^{-1}.\]
Let $\gamma-q\alpha_i,\ldots,\gamma-\alpha_i,\gamma, \gamma+\alpha_i,
\ldots,\gamma+p\alpha_i$ be the $\alpha_i$-string through $\gamma$. Then 
$e_i(v_\gamma)=(q+1)v_{\gamma+\alpha_i}$. Applying $\varphi$ yields that 
$[e_i,\be_\gamma^\epsilon] =(q+1)\be_{\gamma+\alpha_i}^\epsilon$ and so 
\[ [\be_\alpha^\epsilon,\be_\beta^\epsilon]=\pm(q+1)\eta\be_{\gamma+
\alpha_i}^\epsilon\eta^{-1}=\pm (q+1) \be_{w(\gamma+\alpha_i)}^\epsilon=
\pm (q+1)\be_{\alpha+\beta}^\epsilon.\]
It remains to note that $w$ maps the above $\alpha_i$-string through $\gamma$
onto the $\alpha$-string through $\beta$. Hence, $q+1=m_\alpha^-(\beta)$.
\end{proof}

\begin{rem} \label{remcb} Let $\bi=\sqrt{-1}\in\C$. For $\alpha
\in\Phi$ we define an element $\hat{e}_\alpha\in \fg$ by 
\[ \hat{e}_\alpha:=\left\{\begin{array}{cl} \be_\alpha^\epsilon & \quad
\mbox{if $\hgt(\alpha)\equiv 0 \bmod 2$},\\ \bi\be_\alpha^\epsilon & \quad
\mbox{if $\hgt(\alpha)\equiv 1 \bmod 2$}.\end{array}\right.\]
Then the relations in Theorem~\ref{chbase} translate to:
$\tilde{\omega}(\hat{e}_\alpha)=-\hat{e}_{-\alpha}$ and 
$[\hat{e}_\alpha,\hat{e}_{-\alpha}]=h_\alpha$ for all $\alpha\in\Phi$;
furthermore, $[\hat{e}_\alpha,\hat{e}_\beta]=\pm m_\alpha^-(\beta)
\hat{e}_{\alpha+\beta}$ if $\alpha,\beta,\alpha+\beta\in\Phi$. Hence,
$\{h_i\mid i\in I\}\cup\{\hat{e}_\alpha\mid \alpha\in\Phi\}$ is a 
{\em Chevalley basis} in the sense of \cite[\S 4.2]{Ca1} or 
\cite[\S 25.2]{H}. (Note that Bourbaki \cite[VIII, \S 4, no.~4]{B} uses
a slightly different sign convention.)
\end{rem}

\begin{exmp} \label{expg2a} Let $i\in I$ and $\beta \in\Phi$. Then the 
above results show that
\begin{alignat*}{2}
N_{\alpha_i, \beta}^\epsilon&=+\epsilon(i)m_i^{-}(\beta) &\qquad &\mbox{if
$\beta+\alpha_i\in \Phi$},\\
N_{-\alpha_i, \beta}^\epsilon&=-\epsilon(i)m_i^{-}(-\beta) &\qquad &\mbox{if
$\beta-\alpha_i\in \Phi$}.
\end{alignat*}
(Indeed, if $\beta+\alpha_i\in \Phi$, then we have $[\be_{\alpha_i}^\epsilon, \be_\beta^\epsilon]=\epsilon(i)
[e_i,\varphi(\be_\beta^\epsilon)]=\epsilon(i)\varphi(e_i(v_\beta))=
\epsilon(i)m_i^-(\beta)\varphi(v_{\beta+\alpha_i})=\epsilon(i)
m_i^-(\beta)\be_{\beta+\alpha_i}^\epsilon$, as required; the argument
for $[\be_{-\alpha_i}^\epsilon, \be_\beta^\epsilon]$ is analogous.) Hence, 
a simple induction on the height of roots, plus the rules in 
Theorem~\ref{chbase}, imply that all $N_{\alpha,\beta}^\epsilon$ are 
determined whenever $\alpha,\beta,\alpha+\beta\in\Phi$.

To give an explicit example, assume that $\Phi$ is of type $G_2$, with 
set of simple roots $\Pi=\{\alpha_1,\alpha_2\}$ such that $(\alpha_1,
\alpha_2^\vee)=-3$ and $(\alpha_2,\alpha_1^\vee)=-1$. Let $\epsilon(1)=1$ 
and $\epsilon(2)=-1$. Then $\be_{\alpha_1}^\epsilon=e_1$ and 
$\be_{\alpha_2}^\epsilon=-e_2$. All $N_{\alpha,\beta}^\epsilon$ are 
determined by Table~\ref{bg2} and by using $\tilde{\omega}
(\be_\alpha^\epsilon)=-\be_{-\alpha}^\epsilon$ for $\alpha\in\Phi$.
\end{exmp}

\begin{table}[htbp] \caption{Structure constants $N_{\alpha,
\beta}^\epsilon$ in type $G_2$ for $\epsilon(1)=1$, $\epsilon(2)=-1$} 
\label{bg2} {\small $[\be_{\alpha_1}^\epsilon,\be_{\alpha_2}^\epsilon]=
\be_{\alpha_1+ \alpha_2}^\epsilon,\quad 
[\be_{\alpha_1}^\epsilon,\be_{\alpha_1+3\alpha_2}^\epsilon]=
\be_{2\alpha_1+3\alpha_2}^\epsilon,\quad
[\be_{\alpha_2}^\epsilon,\be_{\alpha_1+\alpha_2}^\epsilon]=
-2\be_{\alpha_1+2\alpha_2}^\epsilon$,\\[3pt]
$[\be_{\alpha_2}^\epsilon,\be_{\alpha_1+2\alpha_2}^\epsilon]=
-3\be_{\alpha_1+3\alpha_2}^\epsilon,\quad
[\be_{\alpha_1+\alpha_2}^\epsilon, \be_{\alpha_1+2\alpha_2}^\epsilon ] = 
-3\be_{2\alpha_1+ 3\alpha_2}^\epsilon$, \\[3pt]
$[\be_{\alpha_1+ \alpha_2}^\epsilon, \be_{-\alpha_1}^\epsilon ] = 
\be_{\alpha_2}^\epsilon,\quad 
[\be_{\alpha_1+\alpha_2}^\epsilon, \be_{-\alpha_2}^\epsilon] = 
-3\be_{\alpha_1}^\epsilon,\quad 
[\be_{\alpha_1+2\alpha_2}^\epsilon, \be_{-\alpha_2}^\epsilon] = 
-2\be_{\alpha_1+\alpha_2}^\epsilon$, \\[3pt]
$[\be_{\alpha_1+2\alpha_2}^\epsilon,\be_{-\alpha_1-\alpha_2}^\epsilon] 
=-2\be_{\alpha_2}^\epsilon, \;\; 
[\be_{\alpha_1+3\alpha_2}^\epsilon, \be_{-\alpha_2}^\epsilon] =
-\be_{\alpha_1+2\alpha_2}^\epsilon, \;\;
[\be_{\alpha_1+3\alpha_2}^\epsilon, \be_{-\alpha_1-
2\alpha_2}^\epsilon] = \be_{\alpha_2}^\epsilon$, \\[3pt]
$[\be_{2\alpha_1+3\alpha_2}^\epsilon,\be_{-\alpha_1}^\epsilon] = 
\be_{\alpha_1+3\alpha_2}^\epsilon, \quad
[\be_{2\alpha_1+3\alpha_2}^\epsilon, \be_{-\alpha_1-\alpha_2}^\epsilon ] = 
\be_{\alpha_1+2\alpha_2}^\epsilon$, \\[3pt]
$[\be_{2\alpha_1+ 3\alpha_2}^\epsilon,\be_{-\alpha_1-2\alpha_2}^\epsilon] = 
\be_{\alpha_1+\alpha_2}^\epsilon,  \quad 
[\be_{2\alpha_1+3\alpha_2}^\epsilon,\be_{-\alpha_1-3\alpha_2}^\epsilon ] = 
\be_{\alpha_1}^\epsilon$.}
\end{table}

Ringel \cite{Ri} (see also Peng--Xiao \cite{PX}) found an entirely 
different method for fixing the signs in the structure constants, 
starting from {\it any} orientation of the Dynkin diagram of $\Phi$ and 
then using the representation theory of quivers and Hall polynomials. The 
functions $\pm \epsilon$ in Lemma~\ref{remadj0} correspond exactly to the 
two orientations in which every vertex is either a sink or a source. We 
note that, in type $G_2$, the signs obtained as in \cite[p.~139]{Ri} are 
different from those in Table~\ref{bg2}.


\end{document}